\documentclass[12pt,letterpaper]{article}
\usepackage[latin9]{inputenc}
\usepackage{color}
\usepackage{amsmath}
\usepackage{amsthm}
\usepackage{amssymb}
\usepackage{graphicx}
\PassOptionsToPackage{normalem}{ulem}
\usepackage{ulem}
\usepackage[unicode=true,pdfusetitle,
 bookmarks=true,bookmarksnumbered=false,bookmarksopen=false,
 breaklinks=false,pdfborder={0 0 1},backref=page,colorlinks=true]
 {hyperref}
\hypersetup{
 linkcolor=blue, citecolor=blue}

\makeatletter

\pdfpageheight\paperheight
\pdfpagewidth\paperwidth

\providecommand{\tabularnewline}{\\}

\numberwithin{equation}{section}
\numberwithin{figure}{section}
\theoremstyle{plain}
\newtheorem{thm}{\protect\theoremname}[section]
\theoremstyle{plain}
\newtheorem{prop}[thm]{\protect\propositionname}
\theoremstyle{plain}
\newtheorem{lem}[thm]{\protect\lemmaname}

\@ifundefined{date}{}{\date{}}
\usepackage{graphicx}
\usepackage{appendix}

\usepackage{enumitem}
\usepackage[backref=page]{hyperref}
\usepackage{mathabx}

\makeatother

\providecommand{\lemmaname}{Lemma}
\providecommand{\propositionname}{Proposition}
\providecommand{\theoremname}{Theorem}

\begin{document}
\global\long\def\F{\mathcal{F} }%
\global\long\def\Aut{\mathrm{Aut}\chi}%
\global\long\def\C{\mathbb{C}}%
\global\long\def\H{\mathcal{H}}%
\global\long\def\U{\mathcal{U}}%
\global\long\def\P{\mathcal{P}}%
\global\long\def\ext{\mathrm{ext}}%
\global\long\def\hull{\mathrm{hull}}%
\global\long\def\triv{\mathrm{triv}}%
\global\long\def\Hom{\mathrm{Hom}}%

\global\long\def\trace{\mathrm{tr}}%
\global\long\def\End{\mathrm{End}}%

\global\long\def\L{\mathcal{L}}%
\global\long\def\W{\mathcal{W}}%
\global\long\def\E{\mathbb{E}}%
\global\long\def\SL{\mathrm{SL}}%
\global\long\def\R{\mathbb{R}}%
\global\long\def\Z{\mathbf{Z}}%
\global\long\def\rs{\to}%
\global\long\def\A{\mathcal{A}}%
\global\long\def\a{\mathbf{a}}%
\global\long\def\rsa{\rightsquigarrow}%
\global\long\def\D{\mathbf{D}}%
\global\long\def\b{\mathbf{b}}%
\global\long\def\df{\mathrm{def}}%
\global\long\def\eqdf{\stackrel{\df}{=}}%
\global\long\def\ZZ{\mathcal{Z}}%
\global\long\def\Tr{\mathrm{Tr}}%
\global\long\def\N{\mathbb{N}}%
\global\long\def\std{\mathrm{std}}%
\global\long\def\HS{\mathrm{H.S.}}%
\global\long\def\e{\varepsilon}%
\global\long\def\c{\mathbf{c}}%
\global\long\def\d{\mathbf{d}}%
\global\long\def\AA{\mathbf{A}}%
\global\long\def\BB{\mathbf{B}}%
\global\long\def\u{\mathbf{u}}%
\global\long\def\v{\mathbf{v}}%
\global\long\def\spec{\mathrm{spec}}%
\global\long\def\Ind{\mathrm{Ind}}%
\global\long\def\half{\frac{1}{2}}%
\global\long\def\Re{\mathrm{Re}}%
\global\long\def\Im{\mathrm{Im}}%
\global\long\def\p{\mathfrak{p}}%
\global\long\def\j{\mathbf{j}}%
\global\long\def\uB{\underline{B}}%
\global\long\def\tr{\mathrm{tr}}%
\global\long\def\rank{\mathrm{rank}}%
\global\long\def\K{\mathcal{K}}%
\global\long\def\hh{\mathcal{H}}%
\global\long\def\h{\mathfrak{h}}%

\global\long\def\EE{\mathcal{E}}%
\global\long\def\PSL{\mathrm{PSL}}%
\global\long\def\G{\mathcal{G}}%
\global\long\def\Int{\mathrm{Int}}%
\global\long\def\acc{\mathrm{acc}}%
\global\long\def\awl{\mathsf{awl}}%
\global\long\def\even{\mathrm{even}}%
\global\long\def\z{\mathbf{z}}%
\global\long\def\id{\mathrm{id}}%
\global\long\def\CC{\mathcal{C}}%
\global\long\def\cusp{\mathrm{cusp}}%
\global\long\def\new{\mathrm{new}}%

\global\long\def\LL{\mathbb{L}}%
\global\long\def\M{\mathbf{M}}%
\global\long\def\I{\mathcal{I}}%
\global\long\def\X{X}%
\global\long\def\free{\mathbf{F}}%
\global\long\def\into{\hookrightarrow}%
\global\long\def\Ext{\mathrm{Ext}}%
\global\long\def\B{\mathcal{B}}%
\global\long\def\Id{\mathrm{Id}}%
\global\long\def\Q{\mathbb{Q}}%

\global\long\def\O{\mathcal{T}}%
\global\long\def\Mat{\mathrm{Mat}}%
\global\long\def\NN{\mathrm{NN}}%
\global\long\def\nn{\mathfrak{nn}}%
\global\long\def\Tr{\mathrm{Tr}}%
\global\long\def\SGRM{\mathsf{SGRM}}%
\global\long\def\m{\mathbf{m}}%
\global\long\def\n{\mathbf{n}}%
\global\long\def\k{\mathbf{k}}%
\global\long\def\GRM{\mathsf{GRM}}%
\global\long\def\vac{\mathrm{vac}}%
\global\long\def\SS{\mathcal{S}}%
\global\long\def\red{\mathrm{red}}%
\global\long\def\V{V}%
\global\long\def\SO{\mathrm{SO}}%
\global\long\def\Gd{\Gamma^{\vee}}%
\global\long\def\fd{\mathrm{fd}}%
\global\long\def\perm{\mathrm{perm}}%
\global\long\def\tos{\xrightarrow{\mathrm{strong}}}%
\global\long\def\HH{\mathbb{H}}%
\global\long\def\T{\mathcal{T}}%
\global\long\def\Jac{\mathsf{Jac}}%
\global\long\def\Ham{\mathsf{Ham}}%
\global\long\def\chit{\chi}%

\vspace{-5in} 
\title{The limit points of the bass notes of arithmetic hyperbolic surfaces}
\author{Michael Magee\\
\emph{\small{}To Peter Sarnak on the occasion of his 70th birthday}}
\maketitle
\begin{abstract}
We prove that the limit points of the bass notes of arithmetic hyperbolic
surfaces are the interval $\left[0,\frac{1}{4}\right]$.{\footnotesize{}\tableofcontents{}}{\footnotesize\par}
\end{abstract}

\section{Introduction}
\begin{thm}
\label{thm:main}Let 
\[
\mathsf{BASS}(\mathsf{arith}_{2})\eqdf\{\,\lambda_{1}(\Lambda\backslash\mathbb{H})\,:\,\Lambda\text{ is an arithmetic lattice in \ensuremath{\PSL_{2}(\R)\,\}.}}
\]
The limit points of $\mathsf{BASS}(\mathsf{arith}_{2})$ are the interval
$\left[0,\frac{1}{4}\right]$. In fact this is also true if we restrict
to $\Lambda\leq\PSL_{2}(\Z)$.
\end{thm}

Here $\lambda_{1}(Y)$ is the bottom of the non-trivial\footnote{A multiplicity-one zero is removed.}
spectrum of the Laplace--Beltrami operator on $Y$. 

This theorem confirms a prediction of Sarnak that is part of his program
`Prescribing the spectra of locally uniform geometries'. This program
was described during his Chern lectures \cite{SarnakChernLecture}
(U.C. Berkeley, Jan.-Feb. 2023) and encompasses, among other things,
the geometry of numbers, the Markoff spectrum, the Oppenheim conjecture,
and results about Euclidean sphere packings in one viewpoint. Let
us now point to the backdrop of Theorem \ref{thm:main}.

\uline{0.} The fact that $0$ is a limit point of $\mathsf{BASS}(\mathsf{arith}_{2})$
follows from an old result due to Randol \cite{Randol} and independently,
Selberg \cite{SelbergFourier}. 

\uline{1.} Selberg's Eigenvalue Conjecture \cite{SelbergFourier}
says that $\frac{1}{4}$ is the only possible value of $\lambda_{1}(\Lambda\backslash\HH)$
when $\Lambda$ is a \uline{congruence} subgroup of $\PSL_{2}(\Z)$.
The best progress towards this conjecture is Kim--Sarnak \cite[App. 2]{KIM}
giving a lower bound of $\frac{975}{4096}$.

\uline{2.} There are obviously only finitely many subgroups of
$\SL_{2}(\Z)$ with a given upper bound on their index and hence covolume
in $\mathbb{H}$. Hence in establishing Theorem \ref{thm:main} one
must have 
\[
\mathrm{vol}\left(\Lambda_{i}\backslash\mathbb{H}\right)\to\infty.
\]
If one extends consideration to all arithmetic lattices in $\PSL_{2}(\R)$
the same is true by a result of Borel \cite{Borel}. 

\uline{3.} It is clear that for non-compact lattices in $\PSL_{2}(\R)$,
$\lambda_{1}(\Lambda\backslash\mathbb{H})\in\left[0,\frac{1}{4}\right]$
since there is continuous spectrum of the Laplacian with support $\big[\frac{1}{4},\infty\big)$
\cite{LP}. For cocompact arithmetic lattices, a result of Huber \cite{Huber}
states $\limsup\lambda_{1}(\Lambda\backslash\HH)\leq\frac{1}{4}$
for any sequence with covolume tending to infinity; combined with
the aforementioned result of Borel, this means any limit point of
$\lambda_{1}$ must be in $\left[0,\frac{1}{4}\right]$.

\uline{4.} The fact that $\frac{1}{4}$ is a limit point of $\mathsf{BASS}(\mathsf{arith}_{2})$
follows from work of Hide and the author \cite{HideMagee}. If $\mathsf{BASS(arith_{2}^{c})}$
denotes the bass notes of compact arithmetic surfaces then $\frac{1}{4}$
is a limit point of $\mathsf{BASS(arith_{2}^{c})}$ by work of Hide,
Louder, and the author \cite{louder2023strongly}. It seems possible
to combine the methods of the current paper with estimates towards
the Selberg conjecture \cite{KIM} (and Jacquet--Langlands \cite{JacquetLanglands})
to prove the limit points of $\mathsf{BASS(arith_{2}^{c})}$ contain
the interval $\left[0,\frac{975}{4096}\right]$. There is somewhat
of a barrier\footnote{It seems to require knowing something like this. For $\Gamma_{2}$
a surface group of genus 2, do there exist a sequence of $\phi_{i}:\Gamma_{2}\to S_{n_{i}}$
which are tangle-free on scale $c(\log n_{i})^{\alpha}$, $c,\alpha>0$,
and such that $\std\circ\phi_{i}:\Gamma_{2}\to\U(n_{i}-1)$ strongly
converge to the regular representation of $\Gamma_{2}$?} to use strong convergence as in this paper to obtain the maximal
interval $[0,\frac{1}{4}]$ as limit points of $\mathsf{BASS(arith_{2}^{c})}$,
this therefore remains a very interesting open problem.

The result corresponding to Theorem \ref{thm:main} for $d$-regular
graphs, $d\geq3$ fixed, in place of arithmetic hyperbolic surfaces
has recently been proved by Alon and Wei \cite{AlonWei}. See $\S\S$\ref{subsec:Related-works}
for a more detailed comparison between these works.

On route to Theorem \ref{thm:main} we prove a theorem that may be
of independent interest. Before stating it we recall some notation.
Given a finite set $S$ and two functions $\xi_{1},\xi_{2}:S\to\{1,-1\}$,
the \emph{Hamming distance} $d_{\Ham}(\xi_{1},\xi_{2})$ between $\xi_{1}$
and $\xi_{2}$ is by definition the number of points in $S$ at which
the functions disagree.

Given a finite-area hyperbolic surface $Y=\Gamma\backslash\HH$, let
\[
\Jac^{(2)}(Y)\eqdf\Hom(\Gamma,\Z/2\Z)
\]
 be the 2-torsion in the Jacobian variety of $Y$. Given any $\chi\in\Jac^{(2)}(Y)$
there is an Hermitian line bundle $\L_{\chi}$ over $Y$ associated
to $\chi$ and with associated Laplacian operator etc; let $\lambda_{1}(\L_{\chi})$
denote the smallest eigenvalue of this operator, if it exists, or
$\frac{1}{4}$ otherwise. Given a finite generating set $\B$ of $\Gamma$,
and $\chi_{1},\chi_{2}\in\Jac^{(2)}(Y)$ let 
\begin{align*}
d_{\Ham}^{\B}(\chi_{1},\chi_{2}) & \eqdf d_{\Ham}\text{(}(\chi_{1}(b):b\in\B),(\chi_{2}(b):b\in\B)).
\end{align*}
Our main technical contribution is the following.
\begin{thm}
\label{thm:continuity}Let $X=\Gamma(2)\backslash\HH$. There is $c>0$
such that the following holds. For any $\lambda_{0}<\frac{1}{4}$,
there is $C(\lambda_{0})>0$ such that if $Y_{n}$ is a uniformly
random degree $n$ cover of $Y$ then with probability tending to
one as $n\to\infty$ there is a generating set $\B_{n}$ of $\pi_{1}(Y_{n})$
such that if $d_{\Ham}^{\B_{n}}(\chi_{1},\chi_{2})=1$ and $\lambda_{1}(\L_{\chi})\leq\lambda_{0}$
then
\[
|\lambda_{1}(\L_{\chi_{1}})-\lambda_{1}(\L_{\chi_{2}})|<C(\lambda_{0})n^{-c\sqrt{\frac{1}{4}-\lambda}}.
\]
\end{thm}

\uline{Remark} While stated as a probabilistic result, the proof
relies only on two deterministic conditions: that $Y_{n}$ is connected,
and that $Y_{n}$ satisfies \textbf{GTF($K\log n$) }of $\S$\ref{sec:Probabilistic-input}
for some uniform $K>0$ not depending on $n$.

We prove that \textbf{GTF($K\log n)$} holds for random covers $Y_{n}$
as above for some uniform $K>0$. This is a geometric version of the
tangle free condition used by Friedman \cite{Friedman}. A geometric
version of this property was proven to hold for compact Weil--Petersson
random surfaces by Monk and Thomas \cite{MonkThomas} (see also \cite{Monk,GLST}).
A very recent and independent preprint of Klukowski and Markovi\'{c}
\cite{KlukowskiMarkovic} proves a related geometric property ---
the `$L$-horoball property' --- of random covers of a cusped hyperbolic
surface.

\subsection{Related works\label{subsec:Related-works} }

\textbf{Graphs. }As we mentioned before, the analog of Theorem \ref{thm:main}
for all $d$-regular graphs was proved by Alon and Wei \cite{AlonWei}.

Our approach certainly has the same overall strategy as Alon--Wei
(a type of discrete continuity between two known end points) but differs
in several ways:
\begin{itemize}
\item We use two--covers rather than flipping edges in graphs. While it
is possible flipping edges could have worked in our setting, using
two-covers provides an alternative method that possibly extends more
easily to more general manifolds e.g. hyperbolic three-manifolds.
We plan to return to this topic. 
\item We produce (doubly) exponentially stronger bounds on the type of continuity
one can obtain. This is at least partly because we relax girth conditions
to tangle-free conditions, but are able to use them in a similar way.
\item Obviously, we have to deal with all complications arising from the
geometric setting, with the especial complication that our surfaces
are not compact (this complication was chosen in favor of having non-free
fundamental groups, which is a separate issue in the compact case).
\end{itemize}
In the setting of $d$-regular graphs, one could also ask about the
bass notes of only \uline{arithmetic}\emph{ }graphs (those that
arise as locally symmetric spaces attached to arithmetic lattices
in e.g. $\mathrm{GL_{2}}(\Q_{p}))$. This can likely be dealt with
by covering spaces as we do here. In this case it may also be useful
to use two-covers as in this paper instead of flips.

\textbf{Surfaces. }In the context of (compact) hyperbolic orbifolds,
there is the analog question of Theorem \ref{thm:main} for \uline{all}
compact hyperbolic orbifolds. This has been studied by Kravchuk, Mazac,
and Pal \cite{Kravchuk:2021akc} who obtain rigorous bounds on Laplacian
spectra via a framework inspired by the conformal bootstrap in conformal
field theory in conjunction with linear programming. This allows them
to formulate a precise conjecture: if 
\[
\mathsf{BASS}(\mathsf{hyp}_{2}^{c})\eqdf\{\,\lambda_{1}(\Lambda\backslash\mathbb{H})\,:\,\Lambda\text{ a cocompact lattice in \ensuremath{\PSL_{2}(\R)\,\}}},
\]
then the limit points of $\mathsf{BASS}(\mathsf{hyp}_{2}^{c})$ should
be an interval 
\[
[0,15.79023...]
\]
where the right hand endpoint is $\lambda_{1}$ of a specific hyperbolic
orbifold. In \emph{(ibid.)} this conjecture is reduced \cite[Thm. 4.3]{Kravchuk:2021akc}
to four inequalities concerning $\lambda_{1}$ of specific orbifolds
(that should be numerically check-able) plus a conjecture about the
global maximum of $\lambda_{1}$ on the moduli space of genus $0$
orbifolds with orbifold points of orders $(2,2,2,3)$.

\textbf{Discrete bass note spectra.} While the current work addresses
the \emph{limit points }of $\mathsf{BASS}(\mathsf{arith}_{2})$, and
completes the picture for non-compact surfaces\footnote{At least, without beginning a discussion about cusp forms.},
in the case of compact surfaces the \emph{discrete }points of $\mathsf{BASS}(\mathsf{arith}_{2}^{c})$
and $\mathsf{BASS}(\mathsf{hyp}_{2}^{c})$ are also of basic interest.
Sarnak conjectures \cite{SarnakChernLecture} that there are infinitely
many such discrete points of $\mathsf{BASS}(\mathsf{arith}_{2}^{c})$:
concretely, there are infinitely many values $\lambda=\lambda_{1}(Y)>\frac{1}{4}$
with $Y$ arithmetic compact. This remains open.

In the setting of $\mathsf{BASS}(\mathsf{hyp}_{2}^{c})$, in contrast
to the arithmetic setting, Kravchuk, Mazac, and Pal \cite{Kravchuk:2021akc}
explain that there should be three (precisely known) discrete points
of $\mathsf{BASS}(\mathsf{hyp}_{2}^{c})$ and this is proved modulo
resolving the same issues as previously discussed in the context of
\emph{(ibid.).}

\subsection{Acknowledgments}

We thank Noga Alon, Charles Bordenave, Benoit Collins, Vlad Markovi\'{c},
Doron Puder, Peter Sarnak, and Joe Thomas for conversations about
this work.

\uline{Funding}: This material is based upon work supported by
the National Science Foundation under Grant No. DMS-1926686. This
project has received funding from the European Research Council (ERC)
under the European Union\textquoteright s Horizon 2020 research and
innovation programme (grant agreement No 949143).

\subsection*{Notation}

If $S$ is a finite set we write $\ell_{0}^{2}(S)$ for the set of
$\ell^{2}$ functions on $S$ that are orthogonal to constant functions.
We use the notation $[n]\eqdf\{1,\ldots,n\}$ for $n\in\N$.

\section{Set up\label{sec:Set-up}}

Let $\Gamma\eqdf\Gamma(2)\subset\PSL_{2}(\Z)$ be the (projective
image of the) principal congruence subgroup of level 2. This $\Gamma$
is free of rank 2 on the generators
\[
a\eqdf\left(\begin{array}{cc}
1 & 2\\
0 & 1
\end{array}\right),\quad b\eqdf\left(\begin{array}{cc}
1 & 0\\
2 & 1
\end{array}\right),
\]
and $X\eqdf\Gamma\backslash\mathbb{H}$ is a hyperbolic 3-times punctured
sphere.

\begin{figure}
\begin{centering}
\-\-\-\includegraphics{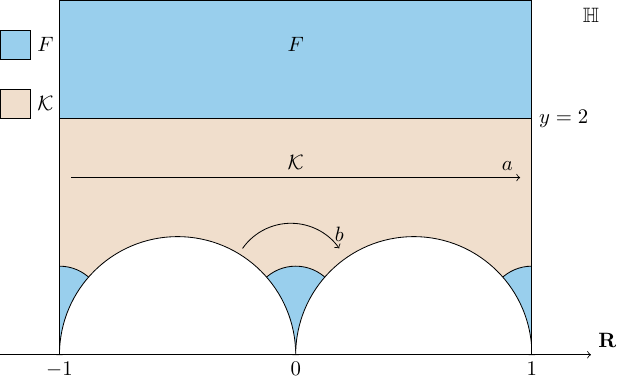}
\par\end{centering}
\centering{}\caption{$F$, fundamental domain for $\Gamma$\label{fig:,-fundamental-domain}}
\end{figure}

We use $\phi$ to denote a homomorphism $\phi:\Gamma\to S_{n}$. Given
$\phi$, $\Gamma$ acts on $\mathbb{H}\times[n]$ by $\gamma(z,x)\eqdf(\gamma z,\phi(\gamma)x)$.
We denote by
\[
X_{\phi}\eqdf\Gamma\backslash_{\phi}\left(\mathbb{H}\times[n]\right)
\]
the quotient by this action. The resulting $X_{\phi}$ is a degree
$n$ Riemannian covering space of $X$ and as such a hyperbolic surface.
If $X_{\phi}$ is connected --- later we will ensure this --- then
\[
X_{\phi}\cong\Gamma_{\phi}\backslash\mathbb{H}
\]
where $\Gamma_{\phi}$ is the stabilizer subgroup
\[
\Gamma_{\phi}\eqdf\mathrm{Stab}_{\phi}(1).
\]
Let 
\[
\mu_{2}=\mu_{2}(\C)\eqdf(\{\pm1\},\times)
\]
 be the multiplicative group of the complex square roots of unity.

Now consider a general homomorphism $\chit:\Gamma_{\phi}\to\mu_{2}$.
If $\chi$ is non-trivial, it defines a (Galois) double cover $X_{\phi,\chit}$
of $X_{\phi}$ via
\[
X_{\phi,\chit}\eqdf\ker(\chi)\backslash\HH.
\]
Smooth functions on $X_{\phi,\chit}$ that are odd under the deck
transformation involution correspond to smooth sections of the Hermitian
line bundle 
\[
\C\to\L_{\chit}\to X_{\phi}
\]
associated to $\chit$. This correspondence intertwines the following
table of differential operators (acting on local trivializations of
sections).
\begin{center}
\begin{tabular}{|c|c|}
\hline 
$C_{\mathrm{odd}}^{\infty}(X_{\phi,\chit})$ & $C^{\infty}(X_{\phi};\L_{\chit})$\tabularnewline
\hline 
\hline 
$\nabla_{X_{\phi,\chit}}$ & $\nabla_{X_{\phi}}$\tabularnewline
\hline 
$\Delta_{X_{\phi,\chit}}$ & $\Delta_{X_{\phi}}$\tabularnewline
\hline 
\end{tabular}
\par\end{center}

Let $F$ be an ideal hyperbolic quadrilateral in $\mathbb{H}$ with
vertices at $-1,0,1,\infty$. Then $F$ is a fundamental domain for
$\Gamma(2)$ with side pairings given by $a,b$ as depicted in Figure
\ref{fig:,-fundamental-domain}. Let $K$ denote the bounded component
of $X$ obtained by cutting along the three closed horocycles\footnote{By the collar lemma for non-compact surfaces \cite[Lemma 4.4.6]{BuserBook},
these horocycles are disjoint. } of length one. Let $\K\subset F$ denote the subset corresponding
to $K$. Let $K_{\phi}$ denote the preimage of the core $K$ in $X_{\phi}$.
This is bounded by some closed horocycles that are lifts of the horocycles
bounding $K$. This $K_{\phi}$ is tiled by copies of $\K$.

The complement of $K_{\phi}$ in $X_{\phi}$ is a union of disjoint
cusps. For $z\in X_{\phi}\backslash K_{\phi}$, let $y(z)$ denote
the coordinate of 
\[
z=x+iy(z)
\]
in a model of the universal cover of the cusp in the hyperbolic upper
half plane where the bounding horocycle of the cusp is covered by
the line $\Im(z)\equiv1$. For example, if $y(z)=1$, then $z\in\partial K_{\phi}$.
We call this the \emph{standard model of the cusp.}

\section{Probabilistic input\label{sec:Probabilistic-input}}

We say that a family of events hold asymptotically almost surely (a.a.s.)
if they hold with probability tending to one as $n\to\infty$. We
will prove that for uniformly random 
\[
\phi\in\Hom(\Gamma,S_{n})
\]
the following simultaneously happen a.a.s.
\begin{description}
\item [{Connected}] $X_{\phi}$ is connected. The fact that $X_{\phi}$
is connected a.a.s. follows from a result of Dixon \cite{Dixon} that
states that two uniformly random and independent elements of $S_{n}$
generate $A_{n}$ a.a.s. This implies that the image of $\phi$ is
a.a.s. transitive on $[n]$, and this is what is required to ensure
$X_{\phi}$ connected.
\item [{Spectral~gap}] For any non-trivial $\theta\in\Hom(\Gamma,\mu_{2})$,
let $\chi(\theta)$ denote the restriction of $\theta$ to $\Gamma_{\phi}$.
We have a.a.s., for any non-trivial such $\theta$
\begin{equation}
\lambda_{1}(X_{\phi,\chi(\theta)})\geq\frac{1}{4}-o_{n\to\infty}(1).\label{eq:sg_eq}
\end{equation}
\uline{Remark.} This relies crucially on the special fact that
any connected degree two covering space $X_{\theta}$ of $X$ has
$\lambda_{1}(X_{\theta})\geq\frac{1}{4}$. This, as well as (\ref{eq:sg_eq})
will be established in $\S$\ref{sec:Probabilistic-spectral-gap}.
\item [{Geometric~tangle~freeness}] \textbf{GTF($c\log n)$ }holds for
some $c>0$ not depending on $n$:\\
\textbf{GTF($R$): }For any $z\in X_{\phi}$, the fundamental group
$\pi_{1}(B_{X_{\phi}}(z,R))$ is cyclic (possibly trivial). \\
This will be established in $\S$\ref{sec:-geometric-tangle}.
\end{description}
Since there only finitely many conditions above, it obviously suffices
to prove each occur a.a.s. 

\section{Probabilistic spectral gap\label{sec:Probabilistic-spectral-gap}}

We assume $X_{\phi}$ is connected here since we already explained
this holds a.a.s.

First we establish the needed spectral properties of degree two covers
of $X$. There is a one-to-one correspondence between non-trivial
$\theta\in\Hom(\Gamma,\mu_{2})$ and connected degree two covers $X_{\theta}$
of $X$ via taking kernels. Any non-trivial such $\theta$ factors
as
\[
\Gamma\to\Gamma/[\Gamma,\Gamma]\to H_{1}(\Gamma,\mu_{2})\cong\mu_{2}\times\mu_{2}\to\mu_{2}
\]
 where the first two maps are canonical. Consider $\Gamma(4)$, the
subgroup of $\Gamma$ consisting of elements of $\Gamma$ congruent
to the class of $\pm I$ modulo 4. It is not hard to check that the
quotient 
\[
\Gamma\to\Gamma/\Gamma(4)
\]
 agrees with the map $\Gamma\to H_{1}(\Gamma,\mu_{2})$ above and
hence any connected degree two cover $X_{\theta}$ of $X$ is an intermediate
cover
\[
X(4)=\Gamma(4)\backslash\HH\to X_{\theta}\to\HH.
\]
By a result of Huxley \cite[\S 6]{Huxley}, $\lambda_{1}(X(4))\geq\frac{1}{4}$.
This implies $\lambda_{1}(X_{\theta})\geq\frac{1}{4}$.

Now take an arbitrary non-trivial $\theta\in\Hom(\Gamma,\mu_{2})$
and let $\chi(\theta)$ be the restriction of $\theta$ to $\Gamma_{\phi}$.
Hide and the author \cite[Thm. 1]{HideMagee} proved that since $\lambda_{1}(X)\geq\frac{1}{4}$
(by Huxley \cite{Huxley} again), a.a.s.
\[
\lambda_{1}(X_{\phi})\geq\frac{1}{4}-o_{n\to\infty}(1).
\]
We explain how to generalize this result to obtain the same for $X_{\phi,\chi(\theta)}$.
As we explained in $\S$\ref{sec:Set-up}, odd functions on $X_{\phi,\chi(\theta)}$
correspond to sections of $\L_{\chi(\theta)}$ in a Laplacian-respecting
way. Sections of $\L_{\chi(\theta)}$ correspond to sections of the
Hermitian vector bundle 
\[
V_{\theta\otimes\rho_{\phi}}\to X
\]
associated to the unitary representation $\theta\otimes\rho_{\phi}$,
where $\rho_{\phi}$ is the composition
\[
\rho_{\phi}:\Gamma\xrightarrow{\phi}\text{\ensuremath{\mathrm{Perms}}}([n])\to\U(\ell^{2}([n])).
\]
Recall that if $S$ is a finite set, $\ell_{0}^{2}(S)$ denotes the
orthocomplement in $\ell^{2}(S)$ of the constant functions with respect
to the standard inner product. We have a splitting $\rho_{\phi}\cong\mathbf{1}\oplus\rho_{\phi}^{0}$
where
\[
\rho_{\phi}^{0}:\Gamma\xrightarrow{\phi}\text{\ensuremath{\mathrm{Perms}}}([n])\to\U(\ell_{0}^{2}([n])).
\]
This means 
\[
\theta\otimes\rho_{\phi}\cong\theta\oplus\left(\theta\otimes\rho_{\phi}^{0}\right).
\]
Sections of the associated bundle to the $\theta$ direct summand
above correspond exactly to odd functions on $X_{\phi,\chi(\theta)}$
that are lifts from odd functions on $X_{\theta}$. This only yields
eigenvalues $\geq\frac{1}{4}$ by previous remarks on the spectrum
of $X_{\theta}$. Thus the only sections of $\L_{\chi(\theta)}$ that
can correspond to eigenvalues of the Laplacian below $\frac{1}{4}$
correspond to sections of the sub-Hermitian vector bundle over $X$
associated to $\theta\otimes\rho_{\phi}^{0}$. We now use two facts
from operator algebras.

\uline{A.} If $\rho_{\phi}^{0}$ converge strongly in probability
to the regular representation $\lambda_{\Gamma}$ of $\Gamma$ (as
they do in the current case by the main result of Bordenave--Collins
\cite{BordenaveCollins}) then for any finite dimensional representation
$\theta,$ $\theta\otimes\rho_{\phi}^{0}$ converge strongly in probability
to 
\[
\theta\otimes\lambda_{\Gamma}.
\]
This is known as `matrix amplification' e.g. \cite[\S 9]{HaagerupThr}.

\uline{B.} On the other hand, by Fell's absorption principle, $\theta\otimes\lambda_{\Gamma}\cong\lambda_{\Gamma}$.
So the upshot is that in this case, $\theta\otimes\rho_{\phi}^{0}$
converges strongly to $\lambda_{\Gamma}$. 

Now the method of \cite{HideMagee}, mutatis mutandis, implies that
if $\lambda_{1}(\L_{\chi(\theta)})$ denotes the bottom of the spectrum
of the Laplacian on sections, then a.a.s.
\[
\lambda_{1}(\L_{\chi(\theta)})\geq\frac{1}{4}-o_{n\to\infty}(1).
\]
By the previous discussion, this gives a.a.s.
\[
\lambda_{1}(X_{\phi,\chi(\theta)})\geq\frac{1}{4}-o_{n\to\infty}(1),
\]
i.e. \textbf{Spectral gap} holds a.a.s.

\section{Geometric tangle freeness\label{sec:-geometric-tangle}}

Let $\epsilon>0$ be a constant to be chosen later on. Recall for
$z\in X_{\phi}\backslash K_{\phi}$ the notation $y(z)$. If $y(z)\geq n^{\epsilon}$
then the distance from $y$ to $\partial K_{\phi}$ is at least
\[
d(y,\partial K_{\phi})\geq\epsilon\log n.
\]
This means that $B_{X_{\phi}}(y,\epsilon\log n)$ has at most cyclic
$\pi_{1}$ and hence \textbf{GTF($\epsilon\log n$) }holds. So we
may assume $y(z)<n^{\epsilon}$ henceforth --- our constant $c$
will be at most whatever $\epsilon$ we choose.

Let $K_{\phi}(\epsilon\log n)$ (resp. $K(\epsilon\log n)$) be the
points of $X_{\phi}$ (resp. $X$) at distance less than $\epsilon\log n$
from $K_{\phi}$ (resp. $K$). So we currently assume $z\in K_{\phi}(\epsilon\log n)$.
Fix some $o\in K$. Since $K$ is compact there is some $R_{0}>0$
such that $K\subset B_{X}(o,R_{0})$, hence
\begin{equation}
K(\epsilon\log n)\subset B_{X}(o,R_{0}+\epsilon\log n).\label{eq:ball-cover}
\end{equation}

The fiber of $o$ under the covering map has representatives 
\begin{align*}
[(o,i)] & \in\Gamma\backslash_{\phi}\left(\mathbb{H}\times[n]\right),\quad i\in[n].
\end{align*}
Suppose that $z\in K_{\phi}(\epsilon\log n)$, which is tiled by copies
of $K(\epsilon\log n)$. Let $i$ be such that $[(o,i)]$ is in the
same $K(\epsilon\log n)$-tile as $z$. Then from (\ref{eq:ball-cover}),
\[
z\in B_{X_{\phi}}([(o,i)],R_{0}+\epsilon\log n).
\]

Let $R>0$ be a parameter. If $B_{X_{\phi}}(x,R)$ has larger than
cyclic $\pi_{1}$, there are two simple non-trivial geodesic loops
beginning and ending at $x$ --- not necessarily smoothly joining
at $x$ --- of length $\leq4R$. This means there are two simple
non-trivial geodesic loops beginning and ending at $[(o,i)]$ of length
$\le4R+2R_{0}+2\epsilon\log n$.

Any such simple geodesic loop covers a simple geodesic loop in $X$
beginning and ending at $o$. This loop represents an element $\gamma\in\pi_{1}(X,o)$.
It follows that $i$ is a fixed point of $\phi(\gamma)$ and 
\[
d(o,\gamma o)\leq4R+2R_{0}+2\epsilon\log n.
\]
 In the situation of the previous paragraph, we obtained:
\begin{itemize}
\item a number $i\in[n]$
\item two non-identity elements $\gamma_{1},\gamma_{2}\in\Gamma$ that generate
a free rank two group $H\leq\Gamma=\pi_{1}(X,o)$, such that $i$
is a common fixed point of $\phi(H)$ and $d(o,\gamma_{i}o)\leq4R+2R_{0}+\epsilon\log n$.
\end{itemize}
We will now show that the above a.a.s. never happens: note that we
have majorized the event of geometric tangles by an event depending
only on elements of $\Gamma$ and $\phi$.

\uline{Lattice point count}

Consider that $K$ has finite diameter $D$ and volume $V$. If $\gamma o\in B_{\mathbb{H}}(o,r)$
for $\gamma\in\Gamma$ then $\gamma\K\subseteq B_{\mathbb{H}}(o,r+D)$,
and since $B_{\mathbb{H}}(o,r+D)$ has volume $\asymp e^{r+D}\asymp e^{r}$,
there are at most $\asymp V^{-1}e^{r}\asymp e^{r}$ such $\gamma$.
In summary,
\begin{equation}
\left|\left\{ \,\gamma\in\Gamma\,:\,d(o,\gamma o)\leq r\,\right\} \right|\ll e^{r}.\label{eq:lattice_point_count}
\end{equation}

\uline{Word lengths}

Let $\gamma\in\Gamma$ such that $d(o,\gamma o)\leq r$. We wish to
bound the word length of $\gamma$ with respect to the generators
$a,b$, or what is the same, the number of times the geodesic arc
in $\mathbb{H}$ from $o$ to $\gamma o$ cuts a side of the tiling
of $\mathbb{H}$ given by $\Gamma$-translates of $F$.

Fix $\gamma$ as above and let $\alpha$ denote this aforementioned
geodesic arc. We subdivide $\alpha$ according to times when it leaves
and enters
\[
\mathfrak{K}\eqdf\bigcup_{\gamma'\in\Gamma}\gamma'\overline{\K}.
\]
By perturbing $o$ an infinitesimal amount we can assume $\alpha$
does not intersect any `corner' in $\partial\mathfrak{K}\cap\bigcup_{\gamma'\in\Gamma}\gamma'\partial F$
--- this is not really an important point, only a technical one.

Any thus obtained subsegment $\alpha'$ of $\alpha$ that is \emph{outside}
$\mathfrak{K}$ intersects $\bigcup_{\gamma'\in\Gamma}\gamma'\partial F$
$\leq\exp(2\ell_{\mathbb{H}}(\alpha'))$ times by direct calculation
in the hyperbolic metric. 

Any subsegment $\alpha'$ that is \emph{inside }$\mathfrak{K}$ intersects
$\bigcup_{\gamma'\in\Gamma}\gamma'\partial F$ at most $1+C\ell_{\mathbb{H}}(\alpha')\leq\exp(C\ell_{\mathbb{H}}(\alpha'))$
times for some $C\geq2$. In total then the number of intersections
is 
\[
\leq\sum_{\alpha'}\exp(C\ell_{\mathbb{H}}(\alpha'))\leq\prod_{\alpha'}\exp(C\ell_{\mathbb{H}}(\alpha'))=\exp(C\ell_{\mathbb{H}}(\alpha)).
\]
In summary, there is $C\geq2$ such that $d(o,\gamma o)\leq r$ implies
that the word length of $\gamma$ is at most $e^{Cr}$.

\uline{Probabilities of common fixed points}

Now fix $\epsilon>0$ such that 
\begin{equation}
6C\epsilon,2\epsilon<\frac{1}{4}.\label{eq:epsilon1}
\end{equation}

Let $\H(R,n)$ denote the collection of rank two subgroups of $\Gamma$
generated by two elements $\gamma_{1}$ and $\gamma_{2}$ with $\ensuremath{d(o,\gamma_{i}o)\leq4R+2R_{0}+\epsilon\log n}$.
Previously we proved \textbf{GTF($R$)}\textbf{\emph{ }}was majorized
by the event that some element of $\H(R,n)$ has a common fixed point
in $[n]$ under $\phi$. Let $\mathrm{fix}_{H}(\phi)$ denote the
number of common fixed points of $\phi(H)$ and suppose 
\[
n\geq e^{6C(4R+2R_{0})}.
\]
Then Proposition \ref{prop:expected_fixed_points_bound} implies that
for any $H\in\H(R,n)$
\[
\E_{n}[\mathrm{fix}_{H}]\leq\frac{e^{6C(4R+2R_{0}+\epsilon\log n)}}{n}\stackrel{\eqref{eq:epsilon1}}{\ll}\frac{n^{\frac{1}{4}}e^{C'R}}{n}
\]
for some $C'>0$. Then
\begin{align*}
 & \E_{n}[\#\text{ common fixed points of subgroups \ensuremath{H} for \ensuremath{H\in\H(R)]}}\\
 & \leq\sum_{H\in\H(R)}\E_{n}[\mathrm{fix}_{H}]\\
 & \ll\sum_{H\in\H(R)}\frac{e^{C'R}}{n^{\frac{3}{4}}}\\
 & \stackrel{\eqref{eq:lattice_point_count}}{\ll}e^{2(4R+2R_{0}+\epsilon\log n)}\frac{e^{C'R}}{n^{\frac{3}{4}}}\\
 & \stackrel{\eqref{eq:epsilon1}}{\ll}\frac{e^{C''R}}{n^{\frac{1}{2}}}
\end{align*}
for some $C''>0$. Now take
\[
R=\frac{1}{4C''}\log n
\]
 to obtain
\[
\E_{n}[\#\text{ common fixed points of subgroups \ensuremath{H} for \ensuremath{H\in\H(R,n)]}}\ll n^{-\frac{1}{4}}
\]
so by Markoff's inequality
\[
\mathbb{P}_{n}[\text{there is any \ensuremath{H\in\H(R,n)} with any common fixed point]\ensuremath{\ll n^{-\frac{1}{4}}}. }
\]
Hence a.a.s. \textbf{GTF($c\log n$) }holds if $c=\min\left(\frac{1}{4C''},\epsilon\right)>0$.

\section{Local bounds for eigensections}

Recall from $\S$\ref{sec:Probabilistic-input} \textbf{GTF($R$):
}For any $z\in X_{\phi}$, the geometric ball $B_{X_{\phi}}(z,R)$
has cyclic fundamental group.

In this section, let $\phi\in\Hom(\Gamma,S_{n})$, assume $X_{\phi}$
is connected, and let $\chi\in\Hom(\Gamma_{\phi},\mu_{2})$. The purpose
of this section is given an eigensection $f$ of $\L_{\chi}$, with
eigenvalue $\lambda$, to give a local bound on $|f|$ that only depends
on $\lambda$ and $\text{\ensuremath{\phi}},$ under the assumption
\textbf{GTF($c\log n$) }for $c$ as in $\S$\ref{sec:Probabilistic-input}. 
\begin{prop}
\label{prop:Linfty-bound}For any $0<c'<c$ the following holds. Let
$\phi,\chi$ be as above. Assume that $\phi$ satisfies \textbf{Connected
}and \textbf{GTF($c\log n$)}. Suppose that $f$ is any smooth section
of $\L_{\chi}$ with eigenvalue $\lambda\in[0,\frac{1}{4})$ and $\|f\|_{L^{2}(X_{\phi})}=1$.
Then:
\begin{enumerate}
\item If $z\in X_{\phi}\backslash K_{\phi}$, and $L$ is the length of
the closed horocycle bounding the cusp containing $z$, then
\[
|f(z)|^{2}\ll_{c}n^{-c\sqrt{\frac{1}{4}-\lambda}}\left(1+\frac{y(z)}{L}\log\left(\frac{2y(z)n^{c}}{L}\right)\right).
\]
\item If $z\in K_{\phi}$ then
\end{enumerate}
\[
|f(z)|^{2}\ll_{c',c}n^{-c'\sqrt{\frac{1}{4}-\lambda}}.
\]
 The implied constants depend on $\Gamma$ and $c',c$ but nothing
else. 
\end{prop}

We will adapt and extend the argument of Gilmore--Le Masson--Sahlsten--Thomas
\cite{GLST} to the present context of cusped surfaces and also explain
why it is uniform over line bundles. We will apply the pre-trace inequality
from \cite[Prop. 5.2]{Gamburd1} that we recall here for the sake
of the reader.
\begin{lem}
\label{lem:pretrace}For $\lambda\in[0,\frac{1}{4}]$ let $r(\lambda)\eqdf i\sqrt{\frac{1}{4}-\lambda}$.
Let $0=\lambda_{0}\leq\lambda_{1}\leq\lambda_{2}\leq\cdots\leq\lambda_{q}\leq\frac{1}{4}$
denote the eigenvalues of the Laplacian on sections of $\L_{\chi}$
that are at most $\frac{1}{4}$, let $f_{j}$ denote their corresponding
$L^{2}$-normalized eigensections, viewed as $\chi$-equivariant functions
for $\Gamma_{\phi}$ acting on $\mathbb{H}$. For any compactly supported
$k:[0,\infty)\to\R$, with \uline{non-negative} Selberg transform
$h$, for any $z\in\mathbb{H}$
\[
\sum_{j\in[q]}h(r(\lambda_{j}))|f_{j}(z)|^{2}\leq\sum_{\gamma\in\Gamma_{\phi}}\chi(\gamma)k(d_{\mathbb{H}}(z,\gamma z)).
\]
\end{lem}

The original version of this inequality does not include the character
$\chi$ but it does not change the proof.
\begin{proof}[Proof of Proposition \ref{prop:Linfty-bound}]
Make the assumptions given in the proposition. Let $k_{t}$ be defined
by
\[
k_{t}(\rho)\eqdf\frac{\mathbf{1}\{\rho\leq t\}}{\sqrt{\cosh t}}.
\]
Then $k_{t}$ is thought of as the kernel $k_{t}(x,y)\eqdf k_{t}(d_{\mathbb{H}}(x,y))$
on $\mathbb{H}\times\mathbb{H}$. As such, the convolution $K_{t}\eqdf k_{t}\ast k_{t}$
makes sense and corresponds to a compactly supported $K_{t}:[0,\infty)\to\R$.
The Selberg transform is a $*$-morphism, so if $H_{t}$ denotes the
Selberg transform of $K_{t}$, we have 
\[
H_{t}=|h_{t}|^{2}\geq0
\]
where $h_{t}$ is the Selberg transform of $k_{t}$. It is proved
in \cite[Proof of Thm. 4.1]{GLST} that for $t\geq3$
\begin{equation}
h_{t}(r(\lambda_{i}))\gg\sinh\left(t\sqrt{\frac{1}{4}-\lambda}\right).\label{eq:lower-bound-h}
\end{equation}
We can also estimate
\begin{align*}
K_{t}(x,y) & =k_{t}\ast k_{t}(x,y)=\frac{1}{\cosh t}\int_{z\in\HH}\mathbf{1}\{d_{\HH}(x,z)\leq t,d_{\HH}(z,y)\leq t\}\\
 & \leq\frac{\mathbf{1}\{d_{\HH}(x,y)\leq2t\}\mathrm{Vol_{\HH}}(B_{\HH}(x,t))}{\cosh t}\ll\mathbf{1}\{d_{\HH}(x,y)\leq2t\}.
\end{align*}
We now apply Lemma \ref{lem:pretrace} with this pair of functions
$K_{t},H_{t}$ and 
\[
t=\frac{1}{2}c\log n
\]
with $c$ as in \textbf{GTF($c\log n$). }Supposing $f$ is an $L^{2}$-normalized
eigensection of $\L_{\chi}$ with eigenvalue $\lambda$ we obtain
for any $z\in\mathbb{H}$ that 
\begin{equation}
|f(z)|^{2}\leq n^{-c\sqrt{\frac{1}{4}-\lambda}}\sum_{\gamma\in\Gamma_{\phi}}\mathbf{1}\{d(z,\gamma z)\leq2t\}.\label{eq:rhs-pointer}
\end{equation}
Now suppose $z\in X_{\phi}$ is fixed --- the outcome of the following
logic will not depend on $z$.

By \textbf{GTF($c\log n$) }the collection of $\gamma$ for which
the right hand side of (\ref{eq:rhs-pointer}) is not zero lie inside
a (possibly trivial) cyclic subgroup of $\Gamma_{\phi}$. Let $\Gamma'$
denote this cyclic subgroup. Suppose it is not trivial (the trivial
case is even easier).

First suppose the non-trivial elements of this $\Gamma'$ are hyperbolic.
If $d(z,\gamma z)\leq2t$ then the displacement length of $\gamma$
is $\le2t$. But displacement length is a homomorphism
\[
\Gamma'\to(\R,+)
\]
and the generator of $\Gamma'$ has displacement length bounded below
by $\ell_{0}$, the shortest length of a closed geodesic on $X$.
This means the right hand side of (\ref{eq:rhs-pointer}) is at most
\[
n^{-c\sqrt{\frac{1}{4}-\lambda}}\frac{4t+1}{\ell_{0}}\ll n^{-c'\sqrt{\frac{1}{4}-\lambda}}
\]
 for any $0<c'<c$.

Now suppose the non-trivial elements of $\Gamma'$ are parabolic,
and let $\gamma_{0}$ be a generator of $\Gamma'$. We split into
two cases.

\uline{Case 1.} $z$ is not in $K_{\phi}$. 

This means $z$ is in a cusp of $X_{\phi}$ bounded by a horocycle
of length $L$ lying over a component of $\partial K$. After changing
coordinates to the standard model of the cusp (cf. $\S$\ref{sec:Set-up}),
$\gamma_{0}(z)=z+L$ where $L$ is the length of the closed horocycle
in $X_{\phi}$ covered by $\Im(z)\eqdf1$. In this model, $z\in X_{\phi}$
corresponds to $z=x+iy(z)$ with $y(z)\geq1$. In this case, the previous
argument does not work because
\begin{equation}
d_{\mathbb{H}}(z,\gamma_{0}^{n}z)=2\sinh^{-1}\left(\frac{|n|L}{2y(z)}\right)\label{eq:distance-frmula}
\end{equation}
and so $4t+1$ above gets replaced by something exponential in $t$.
But we note for later that $d_{\mathbb{H}}(z,\gamma_{0}^{n}z)\leq2t$
implies
\begin{equation}
n\leq\frac{2y(z)e^{2t}}{L}.\label{eq:n-bound}
\end{equation}

Instead, we depart from \cite{GLST} and re-estimate $K_{t}$. This
is already done in Gamburd \cite[Prop. 5.1]{Gamburd1} and gives
\[
K_{t}(z,w)\ll\frac{e^{t}}{\cosh t}\mathbf{1}\{d_{\mathbb{\HH}}(z,w)\leq2t\}e^{-\frac{1}{2}d_{\HH}(z,w)}\ll\mathbf{1}\{d_{\mathbb{\HH}}(z,w)\leq2t\}e^{-\frac{1}{2}d_{\HH}(z,w)}.
\]
Now $\gamma_{0}^{n}$ with $n\neq0$ contributes to the right hand
side of the pre-trace inequality at most (using (\ref{eq:n-bound})
and (\ref{eq:distance-frmula}))

\[
\mathbf{\ll1}\left\{ |n|\leq\frac{2y(z)e^{2t}}{L}\right\} \frac{y(z)}{|n|L}\ll\frac{y(z)}{L}\mathbf{1}\left\{ |n|\leq\frac{2y(z)e^{2t}}{L}\right\} \frac{1}{|n|}.
\]
Summing over all $n$ (counting one for the identity) gives at most
\[
\ll1+\frac{y(z)}{L}\log\left(\frac{2y(z)e^{2t}}{L}\right)
\]
for the right hand side of pre-trace inequality. Putting things together
with $t=\frac{c}{2}\log n$ gives 
\[
|f(z)|^{2}\leq n^{-c\sqrt{\frac{1}{4}-\lambda}}\left(1+\frac{y(z)}{L}\log\left(\frac{2y(z)n^{c}}{L}\right)\right).
\]

\uline{Case 2.} $z\in K_{\phi}.$ 

Take all lifts of $\partial K$ to $\mathbb{H}$, where they form
a countably infinite collection of disjoint horocircles. Each such
horocircle is tangent to $\partial\mathbb{H}$ at a fixed point of
a parabolic element of $\Gamma$, and this gives a one-to-one correspondence
between infinite cyclic parabolic subgroups and horocircles. Our assumption
on $z$ implies $z$ lifts to a point disjoint from all closed discs
bounded by these horocircles. 

In this case we can still change coordinates in the universal cover
so $\gamma_{0}(z)=z+L$, $L\in\pm\N$, and $z=x+iy$ with $y\leq1$.
Then running the same argument as in the previous case, we get 
\begin{align*}
|f(z)|^{2} & \leq n^{-c\sqrt{\frac{1}{4}-\lambda}}\left(1+\frac{y(z)}{L}\log\left(\frac{2y(z)n^{c}}{L}\right)\right)\ll n^{-c\sqrt{\frac{1}{4}-\lambda}}\log\left(n^{c}\right)\\
 & \ll n^{-c'\sqrt{\frac{1}{4}-\lambda}}.
\end{align*}
\end{proof}

\section{Bumping off the cusps\label{sec:Bumping-off-the-cusps}}

We need the following localization formula due to Ismagilov--Morgan--Simon--Sigal
\cite[Thm. 3.2]{CFKS}. 
\begin{thm}[IMS localization formula]
\label{thm:IMS}Suppose $(M,g)$ is a Riemannian manifold and $\{J_{i}\}_{i\in\I}$
are a family of smooth functions $J_{i}:M\to[0,1]$ such that 
\begin{enumerate}
\item $\sum_{i\in\I}J_{i}^{2}\equiv1$
\item on any compact set $K\subset M$, only finitely many $J_{i}$ are
not zero
\item $\sup_{x\in M}\sum_{i\in\I}|\nabla J_{i}(x)|^{2}<\infty$.
\end{enumerate}
Then 
\[
\Delta=\sum_{i\in\I}J_{i}\Delta J_{i}-\sum_{i\in\I}|\nabla J_{i}|^{2}.
\]
Above, $\Delta$ and $\nabla$ are the (non-negative) Laplace-Beltrami
and gradient operators defined w.r.t. $g$.
\end{thm}

Suppose that we are given $X_{\phi}$ satisfying \textbf{Connected
}and \textbf{GTF($c\log n$). }Given $\chi\in\Hom(\Gamma_{\phi},\mu_{2})$,
let $f$ be an $L^{2}$-normalized section of $\L_{\chi}$ on $X_{\phi}$
with eigenvalue $\lambda$ in $[0,\lambda_{0}]$ with $\lambda_{0}<\frac{1}{4}$.
Let $\CC_{1}$ and $\CC_{2}$ be connected components of $X_{\phi}\backslash K_{\phi}$,
i.e. cusps. Let $L_{1}$ and $L_{2}$ denote the lengths of the closed
horocycles bounding the respective $\CC_{i}$. Fix universal $J:\R\to[0,1]$
smooth and such that 
\begin{align}
J(r) & \equiv1\quad r\geq3,\label{eq:chi1}\\
J(r) & \equiv0\quad r\leq2,\label{eq:chi2}\\
\sqrt{1-J^{2}} & \in C^{\infty}.\label{eq:chi3}
\end{align}

In the standard model of $\CC_{i}$ (cf. $\S$\ref{sec:Set-up}),
let 
\[
J_{i}(z)\eqdf J\left(\frac{y(z)}{L_{i}}\right)
\]
define a smooth cutoff on $\CC_{i}$ --- then $J_{1}$ and $J_{2}$
have disjoint supports. Extend both functions by zero to the whole
of $X_{\phi}$. Let 
\[
J_{0}\eqdf\sqrt{1-\left(J_{1}^{2}+J_{2}^{2}\right)}
\]
and
\[
f'\eqdf J_{0}f.
\]
We aim to prove that $f'$ has close to the same Rayleigh quotient
as $f$. That is, we compare
\[
\frac{\int_{X_{\phi}}|\nabla f|^{2}}{\int_{X_{\phi}}|f|^{2}},\,\,\,\frac{\int_{X_{\phi}}|\nabla f'|^{2}}{\int_{X_{\phi}}|f'|^{2}}.
\]
The estimates involving $|f|^{2}$ and $|f'|^{2}$ integrals will
rely on the following result of Gamburd \cite[Lemma 4.1]{Gamburd1}.
\begin{thm}[Gamburd's Collar Lemma]
\label{thm:gamburd}Suppose $\lambda_{0}\in(0,\frac{1}{4})$. There
exists $c_{\lambda_{0}}>0$ such that for any $\phi,\chi,f$ as above,
using coordinate $y(z)$ in a fixed cusp $\CC_{i}$ of $X_{\phi}$
\[
\int_{y(z)=L_{i}}^{2L_{i}}|f|^{2}\geq c_{\lambda_{0}}\int_{y(z)=2L_{i}}^{\infty}|f|^{2}
\]
\end{thm}

(To obtain this statement from Gamburd's result without thinking about
line bundles, one can apply \cite[Lemma 4.1]{Gamburd1} to $f$ viewed
as a function on $X_{\phi,\chi}$, and work in the double cover of
the cusp.)

It follows from Theorem \ref{thm:gamburd} since in $\CC_{i}$, $J_{0}\equiv1$
in $y(z)\leq2L_{i}$ that
\[
\|f'\|^{2}\geq\int_{y(z)=L_{i}}^{2L_{i}}|f|^{2}\geq c_{\lambda_{0}}\int_{y(z)=2L_{i}}^{\infty}|f|^{2}\geq c_{\lambda_{0}}\|J_{i}f\|^{2}.
\]
Adding this for $i=1,2$, using $\|J_{1}f\|^{2}+\|J_{2}f\|^{2}=1-\|f'\|^{2}$
and rearranging gives
\begin{equation}
\|f'\|^{2}\geq\frac{c_{\lambda_{0}}}{2+c_{\lambda_{0}}}>0.\label{eq:mass-estimate}
\end{equation}

We now start the main argument of this $\S$.
\begin{prop}
\label{prop:Rayleigh_1}With notation as above, if \textbf{GTF($c\log n$)
}holds, then for any $c'<c$
\[
\frac{\langle\Delta f',f'\rangle}{\|f'\|^{2}}-\lambda\ll_{c',\lambda_{0}}n^{-2c'\sqrt{\frac{1}{4}-\lambda}}.
\]
\end{prop}

\begin{proof}
Let 
\[
\mathfrak{J}\eqdf|\nabla J_{0}|^{2}+|\nabla J_{1}|^{2}+|\nabla J_{2}|^{2}\geq0.
\]
This is supported in $X_{\phi}\backslash K_{\phi}$ and in $\CC_{i}$,
$\mathfrak{J}$ is supported on $\{z:2L_{i}\leq y(z)\leq3L_{i}\}$.
By Theorem \ref{thm:IMS} we have 
\begin{align*}
\langle\Delta f',f'\rangle & =\langle\mathfrak{J}f,f\rangle+\langle\Delta f,f\rangle-\langle\Delta J_{1}f,J_{1}f\rangle-\langle\Delta J_{2}f,J_{2}f\rangle\\
 & =\langle\mathfrak{J}f,f\rangle+\lambda-\langle\Delta J_{1}f,J_{1}f\rangle-\langle\Delta J_{2}f,J_{2}f\rangle.
\end{align*}
By lifting, each $J_{i}f$, $i=1,2$, can be viewed as a smooth compactly
supported function on a parabolic cylinder (extending a cusp region
of $X_{\tilde{\phi}})$. As the spectrum of the $L^{2}$ Laplacian
in this cylinder is bounded below by $\frac{1}{4}$ it follows as
in \cite[Lemma 4.2]{HideMagee} that
\[
\langle\Delta J_{i}f,J_{i}f\rangle\geq\frac{1}{4}\|J_{i}f\|_{2}^{2}\quad i=1,2.
\]
Inserting this above and rearranging gives
\begin{align}
0\leq\langle\Delta f',f'\rangle & \leq\langle\mathfrak{J}f,f\rangle+\lambda-\frac{1}{4}\left(\|J_{1}f\|_{2}^{2}+\|J_{2}f\|_{2}^{2}\right)\nonumber \\
 & =\langle\mathfrak{J}f,f\rangle+\lambda\|J_{0}f\|_{2}^{2}+\left(\lambda-\frac{1}{4}\right)\left(\|J_{1}f\|_{2}^{2}+\|J_{2}f\|_{2}^{2}\right)\label{eq:take-mass}\\
 & \leq\langle\mathfrak{J}f,f\rangle+\lambda\|f'\|_{2}^{2}.\label{eq:main-ineq}
\end{align}
Now we estimate
\[
|\langle\mathfrak{J}f,f\rangle|\leq\|\mathfrak{J}\|_{1}\max\left(\sup_{z\in C_{1},y(z)\leq3L_{1}}|f(z)|^{2},\sup_{z\in C_{2},y(z)\leq3L_{2}}|f(z)|^{2}\right)
\]
By Proposition \ref{prop:Linfty-bound} we have 
\[
\sup_{z\in C_{i},y(z)\leq3L_{i}}|f(z)|^{2}\ll_{c}n^{-c\sqrt{\frac{1}{4}-\lambda}}\left(1+\log\left(6n^{c}\right)\right)\ll_{c'}n^{-c'\sqrt{\frac{1}{4}-\lambda}}
\]
for any $c'<c$. Hence
\begin{equation}
|\langle\mathfrak{J}f,f\rangle|\ll_{c'}\|\mathfrak{J}\|_{1}n^{-c'\sqrt{\frac{1}{4}-\lambda}}.\label{eq:imtermediate_cusp_bound}
\end{equation}
 We hence turn to $\|\mathfrak{J}\|_{1}$. In $\CC_{1}$
\begin{align*}
\mathfrak{J} & =|\nabla J_{0}|^{2}+|\nabla J_{1}|^{2}=y^{2}\left(\frac{\partial J_{0}}{\partial y}\right)^{2}+y^{2}\left(\frac{\partial J_{1}}{\partial y}\right)^{2}.
\end{align*}
Since derivatives of $J$ and $\sqrt{1-J^{2}}$ are uniformly bounded
\[
\left(\frac{\partial J_{0}}{\partial y}\right)^{2},\left(\frac{\partial J_{1}}{\partial y}\right)^{2}\ll\frac{1}{L_{1}^{2}}.
\]
Hence 
\[
\mathfrak{J}\ll\frac{y^{2}}{L_{1}^{2}}.
\]
Since $\mathfrak{J}$ is supported on $y$ between $2L_{1}$ and $3L_{1}$
we have 
\[
\|\mathfrak{J}\|_{L^{1}(\CC_{1})}\ll\int_{0}^{L_{1}}\int_{2L_{1}}^{3L_{1}}\frac{y^{2}}{L_{1}^{2}}\frac{dy\wedge dx}{y^{2}}=1.
\]
Exactly the same bound holds in $\CC_{2}$ and hence for $\|\mathfrak{J}\|_{1}$,
so
\[
|\langle\mathfrak{J}f,f\rangle|\ll_{c'}n^{-c'\sqrt{\frac{1}{4}-\lambda}}.
\]
Using this in (\ref{eq:main-ineq}) gives
\[
\langle\Delta f',f'\rangle-\lambda\|f'\|^{2}\ll_{c'}n^{-c'\sqrt{\frac{1}{4}-\lambda}};
\]
now dividing by $\|f'\|^{2}$ and using (\ref{eq:mass-estimate})
gives
\begin{equation}
\frac{\langle\Delta f',f'\rangle}{\|f'\|^{2}}-\lambda\ll_{c',\lambda_{0}}n^{-c'\sqrt{\frac{1}{4}-\lambda}}.\label{eq:rayleigh1}
\end{equation}
\end{proof}

\section{Bumping to zero along tile edge\label{sec:Bumping-to-zero-along-edge}}

Recall that $\K$ is the fixed compact part of the fundamental domain
for $\Gamma$. Let $\mathcal{G}$ denote the set of lengths of geodesic
boundary segments of $\mathcal{K}$. Now let $\EE$ denote a infinite
geodesic edge of the tiling of $X_{\phi}$ by $F$-tiles. This edge
meets one or two cusps. The case of one cusp is similar so suppose
there are two, and call them $\CC_{1}$ and $\CC_{2}$ as in the previous
section bounded by closed horocycles of lengths $L_{1}$ and $L_{2}$
respectively.

Let $\kappa\eqdf\frac{1}{2}\sinh^{-1}\left(\frac{1}{8}\right)$ and
$\Omega\eqdf\frac{1}{2}\ell(\EE\cap K_{\phi})\in\frac{1}{2}\G$. We
will use Fermi coordinates in a neighborhood of $\EE$ of the form
\[
(\rho,t)\in[-\kappa,\kappa]\times\left[-(\Omega+\log(4L_{2})),(\Omega+\log(4L_{1}))\right]\xrightarrow{\theta}X_{\phi}
\]
where $\rho$ is the signed distance to $\EE$ and $t$ is the (signed)
arc-length along $\EE$: here we pick a point $p$ on $\EE$ midway
between the two boundary components of $K_{\phi}$ that $\EE$ intersects
for which $t(p)=0$ and we have oriented $\EE$ from $\CC_{2}$ to
$\CC_{1}$ in some way to fix the sign of $\rho$. 

Let $\T$ denote the image of this map $\theta$; the construction
of the constants we use ensure that $\theta$ is a diffeomorphism
to its image. In these coordinates the hyperbolic metric is 
\[
ds^{2}=d\rho^{2}+(\cosh\rho)^{2}dt^{2}.
\]

\begin{figure}
\includegraphics{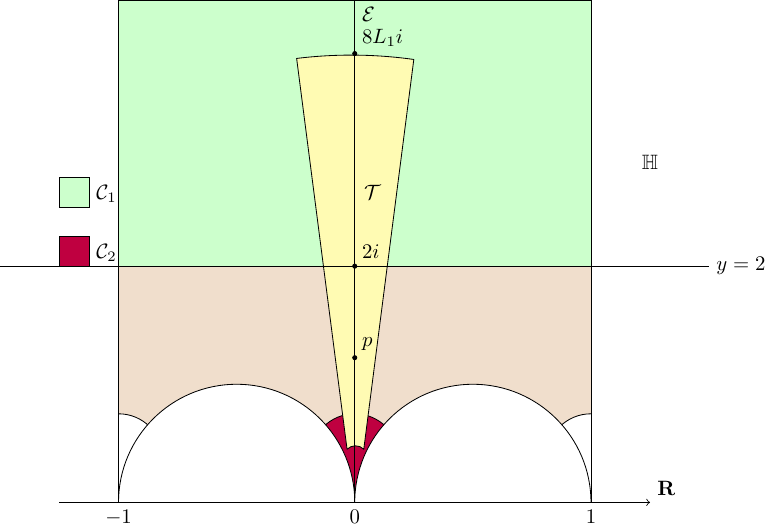}\caption{Illustration of Fermi tube $\mathcal{T}$}

\end{figure}

Given $f$ as in the previous section, let $f'=J_{0}f$ be the function
obtained from the previous section. Let $K_{\phi}^{+}$ denote the
union of $K_{\phi}$ and the regions of $\CC_{i}$ with $y(z)\leq4L_{i}$
in the standard model of each cusp. Then 
\[
K_{\phi}^{+}\supseteq\T.
\]

Let $J$ be as before, satisfying (\ref{eq:chi1})--(\ref{eq:chi3}).
In the above Fermi coordinates in $\T$, let $J^{*}(\rho,t)\eqdf\ensuremath{J\left(\frac{2|\rho|}{\kappa}\right)}$for
$|\rho|\leq\kappa$ and extend by constant value one to a function
on $K_{\phi}^{+}$. Both $J^{*}$ and $\sqrt{1-(J^{*})^{2}}$ are
smooth. Now, $J_{0}J^{*}$ is smooth on $K_{\phi}^{+}$ and extends
(by zero) to a smooth function on all of $X_{\phi}$.

Let $f''\eqdf J_{0}J^{*}f$ and we will write $J^{*}f'$ for this,
understanding what it means (understanding the product only requires
understanding $J^{*}$ in the support of $J_{0}$). 
\begin{prop}
\label{prop:Rayleigh2}With notation as above, for any $c'<c$ and
$\lambda<\lambda_{0}$
\[
\frac{\langle\Delta f'',f''\rangle}{\langle f'',f''\rangle}-\lambda\ll_{c',\lambda_{0}}n^{-c'\sqrt{\frac{1}{4}-\lambda}}.
\]
\end{prop}

\begin{proof}
Let 
\[
Q\eqdf J_{0}^{2}\left(|\nabla J^{*}|^{2}+|\nabla\sqrt{1-J^{*2}}|^{2}\right)\geq0,
\]
$Q$ is supported in $\T$. Moreover, $Q$ is smooth and bounded in
$\T$ (independently of $\EE$). 

The IMS localization formula (Theorem \ref{thm:IMS}) applied to $f'$
tells us that
\begin{align*}
 & \langle\Delta f'',f''\rangle\\
 & =\langle\Delta f',f'\rangle+\langle Qf,f\rangle-\langle\Delta\sqrt{1-J^{*2}}f',\sqrt{1-J^{*2}}f'\rangle\\
 & \leq\langle\Delta f',f'\rangle+\langle Qf,f\rangle.
\end{align*}
Since $Q$ is supported in $\T$,
\begin{equation}
\langle Qf,f\rangle\ll\mathrm{area}(\T)\|f\lvert_{\T}\|_{\infty}^{2}.\label{eq:Q-matrix-ceof-bound}
\end{equation}
The area of $\T$ is 
\begin{align}
\int_{-\kappa}^{\kappa}\int_{-\Omega-\log(4L_{2})}^{\Omega+\log(4L_{1})}\cosh\rho dtd\rho & =\left(2\Omega+\log(4L_{1})+\log(4L_{2})\right)\int_{-\kappa}^{\kappa}\cosh\rho d\rho\nonumber \\
 & \ll\left(2\Omega+\log(4L_{1})+\log(4L_{2})\right)\nonumber \\
 & \ll\log n\label{eq:tube-area-bound}
\end{align}
where the last inequality used both $L_{1},L_{2}\leq n$.

Using Proposition \ref{prop:Linfty-bound} and (\ref{eq:tube-area-bound})
in (\ref{eq:Q-matrix-ceof-bound}) gives for $c'<c$
\[
\langle Qf,f\rangle\ll_{c'',\lambda_{0}}(\log n)n^{-c''\sqrt{\frac{1}{4}-\lambda}}\ll_{c',\lambda_{0}}n^{-c'\sqrt{\frac{1}{4}-\lambda}}.
\]
Hence
\begin{equation}
\langle\Delta f'',f''\rangle-\langle\Delta f',f'\rangle\ll_{c',\lambda_{0}}n^{-c'\sqrt{\frac{1}{4}-\lambda}}.\label{eq:energy-second}
\end{equation}

We also have
\[
\langle f'',f''\rangle=\langle f',f'\rangle-\langle J_{0}\sqrt{1-J^{*2}}f,J_{0}\sqrt{1-J^{*2}}f\rangle
\]
and since $J_{0}\sqrt{1-J^{*2}}$ is supported on $\T$ and universally
bounded, repeating the same arguments as before gives
\begin{equation}
\langle f'',f''\rangle-\langle f',f'\rangle\gg_{c',\lambda_{0}}-n^{-c'\sqrt{\frac{1}{4}-\lambda}}.\label{eq:mass-second}
\end{equation}
Combining the above gives
\begin{align*}
\frac{\langle\Delta f'',f''\rangle}{\langle f'',f''\rangle}-\frac{\langle\Delta f',f'\rangle}{\langle f',f'\rangle}= & \frac{1}{\langle f',f'\rangle}\left(\langle\Delta f'',f''\rangle-\langle\Delta f',f'\rangle\right)\\
 & +\frac{\langle\Delta f'',f''\rangle}{\langle f'',f''\rangle\langle f',f'\rangle}\left(\langle f',f'\rangle-\langle f'',f''\rangle\right)\\
\stackrel{\eqref{eq:mass-estimate}}{\ll_{c',\lambda_{0}}} & \max(\langle\Delta f'',f''\rangle-\langle\Delta f',f'\rangle,0)\\
 & +\frac{\langle\Delta f'',f''\rangle}{\langle f'',f''\rangle}\max(\langle f',f'\rangle-\langle f'',f''\rangle,0)\\
\stackrel{\eqref{eq:energy-second},\eqref{eq:mass-second}}{\ll_{c',\lambda_{0}}} & n^{-c'\sqrt{\frac{1}{4}-\lambda}}\left(1+\frac{\langle\Delta f'',f''\rangle}{\langle f'',f''\rangle}\right).
\end{align*}
This implies now rearranging and using Proposition \ref{prop:Rayleigh_1}
that
\begin{align*}
\frac{\langle\Delta f'',f''\rangle}{\langle f'',f''\rangle}\leq & \left(1+O_{c',\lambda_{0}}\left(n^{-c'\sqrt{\frac{1}{4}-\lambda}}\right)\right)\frac{\langle\Delta f',f'\rangle}{\langle f',f'\rangle}\\
 & +O_{c',\lambda_{0}}\left(n^{-c'\sqrt{\frac{1}{4}-\lambda}}\right)\\
\leq & \left(1+O_{c',\lambda_{0}}\left(n^{-c'\sqrt{\frac{1}{4}-\lambda}}\right)\right)\left(\lambda+O_{c',\lambda_{0}}\left(n^{-c'\sqrt{\frac{1}{4}-\lambda}}\right)\right)\\
 & +O_{c',\lambda_{0}}\left(n^{-c'\sqrt{\frac{1}{4}-\lambda}}\right)\\
 & =\lambda+O_{c',\lambda_{0}}\left(n^{-c'\sqrt{\frac{1}{4}-\lambda}}\right).
\end{align*}
\end{proof}

\section{Proof of Theorems \ref{thm:continuity} and \ref{thm:main}}
\begin{proof}[Proof of Theorem \ref{thm:continuity}]
Suppose that $\phi$ satisfies \textbf{Connected }and\\
 \textbf{GTF($c\log n$) }for $c>0$. Results of previous sections
establish these hold a.a.s. 

Let $G_{\phi}$ denote the $4$-regular Schreier graph constructed
from the side pairing congruences $A,B$ of the fundamental domain
$F$ acting on $[n]$ via $\phi$. Since we assume $X_{\phi}$ is
connected, it is realized as $\Gamma_{\phi}\backslash\HH$ and moreover
$G_{\phi}$ is connected. This graph $G_{\phi}$ encodes the dual
structure of the tiling of $X_{\phi}$ by $\K$-tiles. It has a distinguished
vertex $\otimes$ corresponding to the tile labeled `1' under $X_{\phi}=\Gamma\backslash_{\phi}\left(\HH\times[n]\right)$.
Let $o_{n}$ denote an element of this tile. Since $X_{\phi}$ deformation
retracts to a embedded copy of $G_{\phi}$, we have (via this deformation
retract) $\text{\ensuremath{\pi_{1}(G_{\phi},\otimes)\cong\pi_{1}(X_{\phi},o_{n})}=\ensuremath{\Gamma_{\phi}}}$.

The edges of $G_{\phi}$ are directed and labeled by $A$ and $B$.
Let $T$ denote a spanning tree for $G_{\phi}$. Gluing $\K$-tiles
according to $T$ gives a fundamental domain $F_{\phi}$ for $\Gamma_{\phi}$
in $\mathbb{H}$. Arbitrarily direct the elements of $E(G_{\phi}\backslash T)$
(one could take the previously given directions, for example). For
each $e\in E(G_{\phi}\backslash T)$ there is a unique homotopy class
of loop in $G_{\phi}$ beginning and ending at $\otimes$ that proceeds
(uniquely) to the start point of $e$ in $T$, traverses $e$ in its
given direction, then returns (uniquely) to $\otimes$ in $T$. These
loops gives a free basis $\B\eqdf\{\gamma_{e}:e\in E(G_{\phi}\backslash T)\}$
of the free group $\Gamma_{\phi}\cong\pi_{1}(G_{\phi},\otimes)$.
This is the basis described in Theorem \ref{thm:continuity}. These
elements are the side pairing congruences for $F_{\phi}$.

Let $\chi_{1}$ and $\chi_{2}$ be as in the statement of Theorem
\ref{thm:continuity} with Hamming distance one w.r.t. $\B$. There
is a one-to-one differential-operator-respecting correspondence between
smooth sections of $\L_{\chi_{i}}$ on $X_{\phi}$ and smooth functions
$f$ on $\HH$ such that 
\begin{equation}
f\circ\gamma=\chi_{i}(\gamma)f;\quad\gamma\in\Gamma_{\phi}\label{eq:equivariance}
\end{equation}
as discussed in $\S$\ref{sec:Set-up}. This correspondence implies
also that for any differential operator $D$ on $\HH$ that 
\begin{equation}
D[f\circ\gamma_{e}]=\chi_{i}(\gamma_{e})D[f];\quad e\in E(G_{\phi}\backslash T).\label{eq:equivariance+}
\end{equation}
These functions are uniquely determined by their restriction to $\overline{F_{\phi}}$
where they satisfy (\ref{eq:equivariance+}) for $D$ defined in any
neighborhood of $\partial F_{\phi}$ and such that both sides of (\ref{eq:equivariance+})
can be evaluated with arguments in $\partial F_{\phi}$. (This condition
is intended to make sure that the function on $\overline{F_{\phi}}$
extends to a\emph{ }\uline{smooth} $\Gamma_{\phi}$-periodic function
on $\HH$.) The conditions (\ref{eq:equivariance+}) are an ensemble
of boundary conditions, one for each side-pairing congruence $\gamma_{e}$
of $F_{\phi}$.

If $\chi_{1}$ and $\chi_{2}$ have Hamming distance one with respect
to $\B$, there is a unique $e_{0}\in E(G_{\phi}\backslash T)$ such
that $\chi_{1}$ and $\chi_{2}$ differ at $\gamma_{e_{0}}$. Let
$\EE$ denote the infinite geodesic in $X_{\phi}$ corresponding to
the edge $e_{0}$, and let $f$ be the eigensection for $\L_{\chi_{1}}$
corresponding to the bottom of its spectrum, $\lambda_{1}(\L_{\chi_{1}})$.
Let $f''$ be the section of $\L_{\chi_{1}}$ that is the output of
the previous $\S$\ref{sec:Bumping-off-the-cusps}--\ref{sec:Bumping-to-zero-along-edge}
applied to $f$. 

The fact that $f''$ is zero in a neighborhood of $\EE$ means that,
using the previously mentioned correspondences, if we restrict $f''$
to $F_{\phi}$ and then lift it to a section of $\L_{\chi_{2}},$
we obtain a section $F$ of $\L_{\chi_{2}}$ with (by Proposition
\ref{prop:Rayleigh2})
\[
\frac{\langle\Delta F,F\rangle}{\langle F,F\rangle}\leq\lambda_{1}(\L_{\chi_{1}})+O_{c',\lambda_{0}}\left(n^{-c'\sqrt{\frac{1}{4}-\lambda}}\right).
\]
By the variational principle this implies 
\[
\lambda_{1}(\L_{\chi_{2}})\leq\lambda_{1}(\L_{\chi_{1}})+O_{c',\lambda_{0}}\left(n^{-c'\sqrt{\frac{1}{4}-\lambda}}\right).
\]
\end{proof}
\begin{proof}[Proof of Theorem \ref{thm:main}]
Given any $\eta>0$ --- the `scale' at which we want to find many
$\lambda$ below $\frac{1}{4}$ --- let $c>0$ and $C\left(\frac{1}{4}-\eta\right)>0$
be as in Theorem \ref{thm:continuity}.

Using the results of $\S$\ref{sec:Probabilistic-input} we can find
$n>0$ such that 
\[
C\left(\frac{1}{4}-\eta\right)n^{-c\sqrt{\eta}}<\eta
\]
and $\phi\in\Hom(\Gamma,S_{n})$ such that $X_{\phi}$ is connected,
$X_{\phi}=\Gamma_{\phi}\backslash\HH$, and find a generating set
$\B$ of $\Gamma_{\phi}$ that satisfies the conclusion of Theorem
\ref{thm:continuity} applied with $\lambda_{0}=\frac{1}{4}-\eta$.
Moreover, we can ensure that for some fixed $\theta\in\Hom(\Gamma,\mu_{2})$
\begin{equation}
\lambda_{1}(X_{\phi,\chi(\theta)})\geq\frac{1}{4}-\eta\label{eq:spectral-gap}
\end{equation}
(since \textbf{Spectral Gap }holds a.a.s.). let $\chi_{0}$ denote
the trivial representation of $\Gamma_{0}$, so sections of $\L_{\chi_{0}}$
correspond to functions on
\[
X_{\phi,\chi_{0}}\cong X_{\phi}\sqcup X_{\phi}
\]
that are odd under the deck transformation swapping components. Hence
$\lambda_{1}(\L_{\chi_{0}})=0$ since there is a locally constant
section of $\L_{\chi_{0}}$.

According to the output of Theorem \ref{thm:continuity}, when we
move from $\chi_{0}$ to $\chi_{\theta}$ along a geodesic w.r.t.
$d_{\Ham}^{\B}$, at each step, $\lambda_{1}(\L_{\bullet})$ increases
by at most $\eta.$ But this value begins at zero and ends up at least
$\frac{1}{4}-\eta$, so its values must be $\eta$-dense in $[0,\frac{1}{4}]$.

Finally, if $\chi_{0},\chi_{1},\ldots,\chi_{r}=\chi(\theta)$ are
the result of moving along this geodesic, then
\begin{align*}
\spec(\Delta_{\L_{\chi_{i}}})\subset\spec(\Delta_{X_{\phi,\chi_{i}}}) & =\spec(\Delta_{X_{\phi}})\cup\spec(\Delta_{\L_{\chi_{i}}})\\
 & \subseteq\big[\frac{1}{4}-\eta,\infty\big)\cup\spec(\Delta_{\L_{\chi_{i}}})
\end{align*}
from (\ref{eq:spectral-gap}), so the corresponding $\lambda_{1}(X_{\phi,\chi_{i}})$
are also $\eta$-dense in $\left[0,\frac{1}{4}\right]$.
\end{proof}

\newpage{}

\appendix

\section{Common fixed points of subgroups under random homomorphisms to permutation
groups}

\global\long\def\Q{\mathcal{Q}}%

Let $\Gamma$ denote a finitely generated free group with a fixed
basis $\B$ and $H$ a finitely generated (necessarily free) subgroup.
For $\phi\in\Hom(\Gamma,S_{n})$, let $\mathrm{fix}_{H}$ denote the
number of points in $[n]$ fixed by every element of $\phi(H)\leq S_{n}$.
For a function $F$ on $\Hom(\Gamma,S_{n})$ we write $\E_{n}[F]$
for the expected value of $F$ with respect to the uniform probability
measure on the finite set $\Hom(\Gamma,S_{n})$. Let $\ell_{\B}(H)$
denote the minimum, over all bases $\B'$ of $H$, of the sum of word
lengths of $\B'$ w.r.t. $\B$.
\begin{prop}
\label{prop:expected_fixed_points_bound}If $\rank(H)\geq2$ and $n\geq\ell_{\B}(H)^{3}$,
then
\[
\E_{n}(\mathrm{fix}_{H})\ll\frac{\ell_{\B}(H)^{6}}{n}.
\]
\end{prop}

\begin{proof}
Let $m=\eqdf\ell_{\B}(H)$. In \cite[Lemma 6.4]{PP15}, Puder and
Parzanchevski show that there is a finite set $\Q_{H}$ of rational
functions, depending on $H$, such that:
\begin{itemize}
\item $\E_{n}[\mathrm{fix}_{H}]=\sum_{q\in\Q_{H}}q(n)$ for $n\geq\ell_{\B}(H)$
\item each $q\in\Q_{H}$ is of the form 
\[
q(n)=\frac{(n)_{V(q)}}{\prod_{b\in\B}(n)_{E_{b}(q)}}
\]
where $V(q),E_{b}(q)\leq\ell_{B}(H)$ and $(n)_{a}$ is the falling
Pochhammer symbol. Hence since (e.g. \cite[Lemma 5.23]{MageeNaudPuder})
if $a\leq\frac{1}{2}n$
\begin{equation}
n^{a}\left(1-\frac{a^{2}}{n}\right)\leq(n)_{a}\leq n^{a}\label{eq:poch}
\end{equation}
 we get for $n\geq2\ell_{\B}(H)^{2}$ that 
\[
q(n)\asymp n^{V(q)-\sum_{b\in\B}E_{b}(q)}
\]
in the regime of $n\to\infty$. Let $\chi(q)\eqdf V(q)-\sum_{b\in\B}E_{b}(q)$.
\item For $q\in\Q_{H},$ $\chi(q)\leq1-\rank(H)$.
\end{itemize}
In fact, Puder and Parzanchevski give a much more detailed description
of the set $\Q_{H}$ above, but this will not be needed here. It follows
immediately that for fixed $H$, $\E_{n}[\mathrm{fix}_{H}]=O(n^{\chi(H)})$
but we need a version of this statement that is uniform over (mildly)
varying $H$. It is possible to do this using the detailed description
of $H$ as was done in Puder \cite[\S 5]{PUDER} but we do not need
such fine estimates so can bypass this by analytic methods. A variant
form of the following trick appeared in \cite{MageeNaudPuder} for
a similar purpose.

Suppose $n\geq m^{3}$ and note from (\ref{eq:poch}) that
\[
m^{3}\geq\E_{m^{3}}[\mathrm{fix}_{H}]\gg\sum_{q\in\Q_{H}}(m^{3})^{\chi(q)}.
\]
Then supposing $\rank(H)\geq2$
\[
\E_{n}[\mathrm{fix}_{H}]\ll\sum_{q\in\Q_{H}}n^{\chi(q)}\leq\sum_{\chi=-1}^{-\infty}m^{3(1-\chi)}n^{\chi}\ll\frac{m^{6}}{n}
\]
as required.
\end{proof}
\bibliographystyle{amsalpha}
\bibliography{strong_convergence}

\providecommand{\bysame}{\leavevmode\hbox to3em{\hrulefill}\thinspace}
\providecommand{\MR}{\relax\ifhmode\unskip\space\fi MR }
\providecommand{\MRhref}[2]{%
  \href{http://www.ams.org/mathscinet-getitem?mr=#1}{#2}
}
\providecommand{\href}[2]{#2}
\begin{thebibliography}{GLMST21}

\bibitem[AW23]{AlonWei}
Noga Alon and Fan Wei, \emph{The limit points of the top and bottom eigenvalues
  of regular graphs}, 2023.

\bibitem[BC19]{BordenaveCollins}
C.~Bordenave and B.~Collins, \emph{Eigenvalues of random lifts and polynomials
  of random permutation matrices}, Ann. of Math. (2) \textbf{190} (2019),
  no.~3, 811--875. \MR{4024563}

\bibitem[Bor81]{Borel}
Armand Borel, \emph{Commensurability classes and volumes of hyperbolic
  3-manifolds}, Ann. Sc. Norm. Super. Pisa, Cl. Sci., IV. Ser. \textbf{8}
  (1981), 1--33 (English).

\bibitem[Bus92]{BuserBook}
Peter Buser, \emph{Geometry and spectra of compact {Riemann} surfaces}, Prog.
  Math., vol. 106, Boston, MA: Birkh{\"a}user, 1992 (English).

\bibitem[CFKS87]{CFKS}
H.~L. Cycon, R.~G. Froese, W.~Kirsch, and B.~Simon, \emph{Schr{\"o}dinger
  operators, with application to quantum mechanics and global geometry.},
  Springer {Study} edition. {Texts} and {Monographs} in {Physics}. {Berlin}
  etc.: {Springer}-{Verlag}. ix, 319 pp.; {DM} 56.00 (1987)., 1987.

\bibitem[Dix69]{Dixon}
J.~D. Dixon, \emph{The probability of generating the symmetric group}, Math. Z.
  \textbf{110} (1969), 199--205 (English).

\bibitem[Fri08]{Friedman}
J.~Friedman, \emph{A proof of {A}lon's second eigenvalue conjecture and related
  problems}, Mem. Amer. Math. Soc. \textbf{195} (2008), no.~910, viii+100.
  \MR{2437174}

\bibitem[Gam02]{Gamburd1}
A.~Gamburd, \emph{On the spectral gap for infinite index ``congruence''
  subgroups of {${\rm SL}_2(\bold Z)$}}, Israel J. Math. \textbf{127} (2002),
  157--200. \MR{1900698}

\bibitem[GLMST21]{GLST}
C.~Gilmore, E.~Le~Masson, T.~Sahlsten, and J.~Thomas, \emph{Short geodesic
  loops and {$L^p$} norms of eigenfunctions on large genus random surfaces},
  Geom. Funct. Anal. \textbf{31} (2021), no.~1, 62--110. \MR{4244848}

\bibitem[HM23]{HideMagee}
Will Hide and Michael Magee, \emph{Near optimal spectral gaps for hyperbolic
  surfaces}, Ann. Math. (2) \textbf{198} (2023), no.~2, 791--824 (English).

\bibitem[HT05]{HaagerupThr}
U.~Haagerup and S.~Thorbj{\o}rnsen, \emph{A new application of random matrices:
  {${\rm Ext}(C^*_{\rm red}(F_2))$} is not a group}, Ann. of Math. (2)
  \textbf{162} (2005), no.~2, 711--775. \MR{2183281}

\bibitem[Hub74]{Huber}
H.~Huber, \emph{\"{U}ber den ersten {E}igenwert des {L}aplace-{O}perators auf
  kompakten {R}iemannschen {F}l\"{a}chen}, Comment. Math. Helv. \textbf{49}
  (1974), 251--259. \MR{365408}

\bibitem[Hux85]{Huxley}
M.~N. Huxley, \emph{Introduction to {Kloostermania}}, 1985.

\bibitem[JL70]{JacquetLanglands}
H.~Jacquet and R.~P. Langlands, \emph{Automorphic forms on {${\rm GL}(2)$}},
  Lecture Notes in Mathematics, Vol. 114, Springer-Verlag, Berlin-New York,
  1970. \MR{0401654}

\bibitem[Kim03]{KIM}
H.~H. Kim, \emph{Functoriality for the exterior square of {${\rm GL}_4$} and
  the symmetric fourth of {${\rm GL}_2$}}, J. Amer. Math. Soc. \textbf{16}
  (2003), no.~1, 139--183, With appendix 1 by D. Ramakrishnan and appendix 2 by
  H. H. Kim and P. Sarnak. \MR{1937203}

\bibitem[KM24]{KlukowskiMarkovic}
Adam Klukowski and Vladimir Markovi\'{c}, \emph{Tangle free permutations and
  the {P}utman-{W}ieland property of random covers}, 2024, Preprint,
  arXiv:2402.19018.

\bibitem[KMP21]{Kravchuk:2021akc}
Petr Kravchuk, Dalimil Mazac, and Sridip Pal, \emph{{Automorphic Spectra and
  the Conformal Bootstrap}}.

\bibitem[LM23]{louder2023strongly}
Larsen Louder and Michael Magee, \emph{Strongly convergent unitary
  representations of limit groups}, 2023, arXiv:2210.08953 with Appendix by
  Will Hide and Michael Magee.

\bibitem[LP81]{LP}
P.~D. Lax and R.~S. Phillips, \emph{The asymptotic distribution of lattice
  points in {E}uclidean and non-{E}uclidean spaces}, Functional analysis and
  approximation ({O}berwolfach, 1980), Internat. Ser. Numer. Math., vol.~60,
  Birkh\"{a}user, Basel-Boston, Mass., 1981, pp.~373--383. \MR{650290}

\bibitem[MNP22]{MageeNaudPuder}
Michael Magee, Fr{\'e}d{\'e}ric Naud, and Doron Puder, \emph{A random cover of
  a compact hyperbolic surface has relative spectral gap
  {{\(\frac{3}{16}-\varepsilon\)}}}, Geom. Funct. Anal. \textbf{32} (2022),
  no.~3, 595--661 (English).

\bibitem[{Mon}20]{Monk}
L.~{Monk}, \emph{{Benjamini-Schramm convergence and spectrum of random
  hyperbolic surfaces of high genus}}, Analysis \& PDE, to appear, available at
  arXiv:arXiv:2002.00869 (2020), arXiv:2002.00869.

\bibitem[MT22]{MonkThomas}
Laura Monk and Joe Thomas, \emph{The tangle-free hypothesis on random
  hyperbolic surfaces}, Int. Math. Res. Not. \textbf{2022} (2022), no.~22,
  18154--18185 (English).

\bibitem[PP15]{PP15}
D.~Puder and O.~Parzanchevski, \emph{Measure preserving words are primitive},
  Journal of the American Mathematical Society \textbf{28} (2015), no.~1,
  63--97. \MR{3264763}

\bibitem[Pud15]{PUDER}
D.~Puder, \emph{Expansion of random graphs: new proofs, new results}, Invent.
  Math. \textbf{201} (2015), no.~3, 845--908. \MR{3385636}

\bibitem[Ran74]{Randol}
B.~Randol, \emph{Small eigenvalues of the {L}aplace operator on compact
  {R}iemann surfaces}, Bull. Amer. Math. Soc. \textbf{80} (1974), 996--1000.
  \MR{400316}

\bibitem[Sar23]{SarnakChernLecture}
P.~Sarnak, \emph{Spectra of locally uniform geometries}, Chern Lectures, U.C.
  Berkeley \textit{Available at
  \url{http://publications.ias.edu/sarnak/paper/2728}}, 2023.

\bibitem[Sel65]{SelbergFourier}
A.~Selberg, \emph{On the estimation of {F}ourier coefficients of modular
  forms}, Proc. {S}ympos. {P}ure {M}ath., {V}ol. {VIII}, Amer. Math. Soc.,
  Providence, R.I., 1965, pp.~1--15. \MR{0182610}

\end{thebibliography}

\noindent Michael Magee, \\
Department of Mathematical Sciences, Durham University, Lower Mountjoy,
DH1 3LE Durham, UK\\
IAS Princeton, School of Mathematics, 1 Einstein Drive, Princeton
08540, USA\\
\texttt{michael.r.magee@durham.ac.uk}~\\

\noindent 
\end{document}